\documentclass[11pt]{amsart}
\usepackage{amssymb}
\usepackage{amscd}
\usepackage{color}
\usepackage{url}
\usepackage{hyperref}

\numberwithin{equation}{section}

\theoremstyle{plain}
\newtheorem{theorem}[subsection]{Theorem}
\newtheorem{proposition}[subsection]{Proposition}
\newtheorem{lemma}[subsection]{Lemma}
\newtheorem{corollary}[subsection]{Corollary}

\theoremstyle{definition}

\newtheorem{definition}[subsection]{Definition}

\newcommand{\Z}{\mathbb Z}

\newcommand{\N}{\mathbb N}
\newcommand{\R}{\mathbb{R}}

\DeclareMathOperator*{\Ker}{Ker \;}
\DeclareMathOperator*{\Ran}{Ran \;}

\makeatletter
\def\author@andify{%
  \nxandlist {\unskip ,\penalty-1 \space\ignorespaces}%
    {\unskip {} \@@and~}%
    {\unskip \penalty-2 \space \@@and~}%
}
\makeatother

\title[Bishop operators]{Invariant subspaces for Bishop operators and beyond}

\author{Fernando Chamizo}
\address{Departamento de Matem\'aticas \\ Universidad Aut\'onoma de Ma\-drid \\ Madrid 28049 \\ Spain  and Instituto de Ciencias Matem\'aticas ICMAT (CSIC-UAM-UC3M-UCM), Madrid, Spain}
\email[corresponding author]{fernando.chamizo@uam.es}

\author{Eva A. Gallardo-Guti\'errez}
\address{Departamento de An\'alisis Matem\'atico y Matem\'atica Aplicada, Facultad de Matem\'aticas, Universidad Complutense de Madrid, Plaza de Ciencias 3, 28040 Madrid, Spain and Instituto de Ciencias Matem\'aticas ICMAT (CSIC-UAM-UC3M-UCM), Madrid, Spain}
\email{eva.gallardo@mat.ucm.es}

\author{Miguel Monsalve-L\'opez}
\address{Departamento de An\'alisis Matem\'atico y Matem\'atica Aplicada, Facultad de Matem\'aticas, Universidad Complutense de Madrid, Plaza de Ciencias 3, 28040 Madrid, Spain and Instituto de Ciencias Matem\'aticas ICMAT (CSIC-UAM-UC3M-UCM), Madrid, Spain}
\email{migmonsa@ucm.es}

\author{Adri\'an Ubis}
\address{Departamento de Matem\'aticas\\ Universidad Aut\'onoma de Madrid \\ Madrid 28049 \\ Spain}
\email{adrian.ubis@gmail.com}

\thanks{F. Chamizo and A. Ubis are partially supported by Plan Nacional  I+D grant no. MTM2017-83496-P (Spain); E. A. Gallardo-Guti\'errez and M. Monsalve-L\'opez are partially supported by Plan Nacional  I+D grant no. MTM2016-77710-P (Spain).
The first three authors are also partially supported by
``Severo Ochoa Programme for Centres of Excellence in
R{\&}D'' (SEV-2015-0554). Finally, M. Monsalve-L\'opez also acknowledges support of the grant \emph{Ayudas de la Universidad Complutense de Madrid  para contratos predoctorales de personal investigador en formaci\'on,} ref. no. CT27/16.}

\date{November 27, 2018}

\begin{document}

\begin{abstract}
Bishop operators $T_{\alpha}$ acting on $L^2[0,1)$  were proposed by E. Bishop in the fifties as possible operators which might entail counterexamples for the Invariant Subspace Problem. We prove that all the Bishop operators are biquasitriangular and, derive as a consequence that they are norm limits of nilpotent operators. Moreover, by means of arithmetical techniques along with a theorem of Atzmon,  the set of irrationals $\alpha\in (0,1)$ for which $T_\alpha$ is known to possess non-trivial closed invariant subspaces is considerably enlarged, extending previous results by Davie \cite{DAVIE_inv_subs_bishop}, MacDonald \cite{MACDONALD_inv_bishop_type} and Flattot \cite{flattot}. Furthermore, we essentially show that when our approach fails to produce invariant subspaces it is actually because Atzmon Theorem cannot be applied. Finally, upon applying arithmetical bounds obtained, we deduce local spectral properties of Bishop operators proving, in particular, that neither of them satisfy the \emph{Dunford property}~$(C)$.
\end{abstract}


\maketitle

\section{Introduction}

Perhaps, one of the best-known unsolved problems in Functional Analysis is the \emph{Invariant
Subspace Problem:}

\begin{quotation}
\emph{Does every bounded linear operator on a (separable, infinite-dimensional,
complex) Hilbert space have a non-trivial  closed invariant subspace?}
\end{quotation}

In this regard, one of the earliest and most elegant invariant subspace theorems is the result of von Neumann in the Hilbert space setting (unpublished) and Aronszajn and Smith \cite{Aron1} in the context of Banach spaces which states, in particular, that compact operators have non-trivial closed invariant subspaces. In 1973 operator theorists were stunned by the generalization achieved by Lomonosov \cite{Lomonosov}, who proved one of the most general positive results to provide invariant subspaces, namely: \emph{any linear bounded operator $T$ acting on a Banach space commuting with a non-zero compact operator has a non-trivial closed invariant subspace.} Moreover, $T$ has a non-trivial  \emph{hyperinvariant} closed subspace, that is, a closed subspace which is invariant under every operator in the commutant of~$T$. Accordingly, any linear bounded operator $T$ has a non-trivial  invariant closed subspace if it commutes with a non-scalar operator that commutes with a nonzero compact operator. But, it was not until 1980 that Hadwin, Nordgren, Radjavi, Rosenthal \cite{HNRR} showed the existence of an operator in the Hilbert space setting having non-trivial invariant subspaces to which Lomonosov's Theorem does not apply.\smallskip

In the meantime, two remarkable counterexamples came into scene. Firstly, in 1975 Enflo announced in the S\'eminaire Maurey-Schwarz at the \'Ecole Polytechnique in Paris the existence of a separable Banach space and a linear bounded operator $T$  without non-trivial  closed invariant subspaces; though its publication was delayed for more than ten years \cite{Enflo}.  Then, in 1985, Read~\cite{Re85} constructed a bounded linear operator without non-trivial  closed invariant subspaces in the well-known sequence space $\ell^1$ (see also \cite{Re84} for a previous construction). Indeed, the construction carried over in \cite{Re85} is the first known example of such an operator on any of the classical Banach spaces.\smallskip

For decades a number of authors worked on extending these results to more general classes of operators, and significant progress has been made by developing deep tools in allied areas like Harmonic Analysis, Function Theory or finite dimension Linear Algebra in the framework of Operator Theory. Among different approaches, two have been specially fruitful in order to provide invariant subspaces for a given operator: one coming from the behavior of such operator acting on finite dimensional subspaces leading to the concept of quasitriangular operators.
The other one, mostly based on function theory techniques, consist of developing an ``appropriate" functional calculus which allows to produce hyperinvariant subspaces from the fact that two non-zero functions may have pointwise zero product.\smallskip

Regarding the first approach,  recall that a linear bounded operator $T$  in a separable infinite dimensional Hilbert space $H$  is said to be \emph{quasitriangular} if there exists an increasing sequence $(P_n)_{n=1}^{\infty}$ of finite rank projections converging to the identity $I$ strongly as $n\to \infty$ such that
$$\|TP_n -P_n TP_n \|\to 0,  \quad \mbox{ as } n\to \infty.$$
Based on Aronszajn and Smith's Theorem, Halmos \cite{Halmos} introduced the concept of quasitriangular operators in the sixties to prove the existence of invariant subspaces. It is completely apparent that given a triangular operator in $H$, that is, a linear bounded  operator which admits a representation as an upper triangular matrix with respect to a suitable orthonormal basis, there exists an increasing sequence
$(P_n)_{n=1}^{\infty}$ of finite rank projections converging to the identity $I$ strongly as $n\to \infty$ such that
$$TP_n-P_nTP_n=(I-P_n)TP_n=0, \quad \mbox{ for all } n=0, 1, 2,...$$
Hence, the definition of quasitriangularity says, roughly speaking, that $T$ has a sequence of ``approximately invariant'' finite-dimensional subspaces. Compact operators, operators with finite spectrum, decomposable operators or compact perturbations of normal operators are examples of quasitriangular operators. On the other hand, the shift operator of index one is not quasitriangular; and remarkable results due to Douglas and Pearcy \cite{Douglas-Pearcy} and Apostol, Foias and Voiculescu \cite{AFV-1} yield that the \emph{Invariant Subspace Problem} is reduced to be proved for quasitriangular operators (see Herrero's book \cite{Herrero} for more on the subject).\smallskip

In what the second approach refers, \emph{Beurling algebras} have played an important role in this context. The starting point was a theorem of Wermer \cite{wermer} in 1952 which states that an invertible linear bounded operator $T$ on $H$ such that the series
$$
\sum_{n=-\infty}^{\infty} \frac{\log \|T^n\|}{1+n^2}
$$
converges and its spectrum is not a singleton is either a multiple of the identity or has a non-trivial hyperinvariant closed subspace. A stronger variant was proved by Atzmon (see \cite{Atzmon-1980} and \cite{Atzmon-1984}, for instance). The common feature is the definition of a functional calculus, particularly in \cite{Atzmon-1984} mapping an algebra $\mathcal{A}_{\rho}$ of functions defined on the unit circle $\mathbb{T}$ into $\mathcal{L}(H)$, the Banach algebra of linear bounded operators acting on $H$. For more on the subject we refer to the classical monograph by Radjavi and Rosenthal \cite{RR} and the recent one by Chalendar and Partington \cite{Chalendar-Partington_book}.\smallskip

The main goal of this work is addressing both approaches in the context of Bishop operators. Given an irrational number $\alpha\in (0,1)$, recall that the Bishop operator $T_{\alpha}$ is defined on $L^p[0,1)$, $1\leq p\leq \infty$, by
\[
 T_{\alpha} f(t)= t f(\{t+\alpha\}), \qquad t\in [0,1),
\]
where $\{\,\cdot\,\}$ denotes the fractional part. As explained by Davie \cite{DAVIE_inv_subs_bishop}, these examples were suggested by Bishop as candidates for operators without  non-trivial  closed invariant subspaces. By means of a functional calculus approach,  Davie proved the existence of non-trivial closed \emph{hyperinvariant} subspaces in $L^2[0,1)$ for $T_{\alpha}$ whenever $\alpha$ is a non-Liouville irrational number in $(0,1)$.  Later, subsequent extensions strengthening it due to Blecher and Davie \cite{Blecher-Davie}, MacDonald \cite{MACDONALD_inv_bishop_type}, \cite{MACDONALD_decomposable_weighted} and Flattot  \cite{flattot}  provided a large class of irrationals $\alpha\in (0,1)$ including some Liouville numbers.\smallskip

Our main results in this context will be showing, on one hand, that every Bishop operator $T_{\alpha}$ as well as its adjoint $T_{\alpha}^*$ are quasitriangular operators in $L^2[0,1)$, having therefore a good  approximation by approximately invariant finite-dimensional subspaces. On the other hand, in Theorem \ref{te:main_result} we will extend the class of irrationals $\alpha\in (0,1)$ such that $T_{\alpha}$ has non-trivial closed hyperinvariant subspaces in $L^p[0,1)$ by considering arithmetical techniques which allow to strengthen the analysis of the behavior of certain functions associated to the functional calculus model. Indeed, those Liouville irrationals $\alpha$ escaping the condition set up in Theorem \ref{te:main_result} are so extreme that Theorem \ref{te:atzmon_limit} will show that, essentially, Atzmon Theorem cannot be applied for such irrationals. Roughly speaking, we prove that when our approach fails to produce invariant subspaces it is actually because Atzmon Theorem cannot be applied, what establishes, somehow, the threshold limit in the growth of the denominators of the convergents of those $\alpha$. In some sense, this corroborates an approach to look for invariant subspaces for every $T_\alpha$ based on different functional analytic tools; which will be the goal in the final section.\smallskip
 
On the other hand, observe that by Jarn\'{\i}k-Besicovitch Theorem (see \cite[Section 5.5]{bugeaud}, for instance), Liouville irrationals form a set of vanishing Hausdorff dimension. Nevertheless, it is possible to measure the difference between those cases covered by Davie and Flattot Theorems and  Theorem \ref{te:main_result}, by considering the logarithmic Hausdorff dimension through the use of the family of functions $|\log x|^{-s}$ (instead of the usual $x^s$). With such a dimension, by means of  \cite[Theorem 6.8]{bugeaud}, one can easily deduce that the set of exceptions in Davie, Flattot and our case have dimension $\infty, 4$ and $2$, respectively.
\smallskip

The rest of the manuscript is organized as follows. In Section \ref{Section 2} we introduce some preliminaries and prove that every Bishop operator $T_{\alpha}$ is biquasitriangular in $L^2[0,1)$. In Section \ref{Section 3}, we recall the functional calculus provided by Davie and its extension by Atzmon (a good reference for that is \cite[Chapter 5]{Chalendar-Partington_book}); and construct explicit functions in $L^p[0,1)$ which allow to extend the class of Liouville numbers $\alpha\in (0,1)$ such that $T_{\alpha}$ has non-trivial closed hyperinvariant subspaces. In Section \ref{Section 4}, we show the limits of Atzmon's Theorem approach in the context of Bishop operators. Finally, in Section \ref{Section 5} we discuss some consequences regarding spectral subspaces, which constitute a class of invariant linear manifolds to look for non-trivial closed hyperinvariant subspaces.  We will show, in particular, that $T_{\alpha}$ does not satisfy the \emph{Dunford property $(C)$} in  $L^p[0,1)$ by exhibiting that some spectral subspaces are not closed.

\subsection*{A word about notation} In this paper we employ  a form of Vinogradov's notation. We write $A\ll B$ meaning $|A|\le K |B|$ for some absolute constant $K>0$. Note that in particular we have $A=O(B)$.

\section{Quasitriangular Bishop operators}\label{Section 2}

As mentioned in the introduction, a linear bounded operator $T$ in a separable infinite dimensional Hilbert space $H$  is quasitriangular if there exists an increasing sequence $(P_n)_{n=1}^{\infty}$ of finite rank projections converging to the identity $I$ strongly as $n\to \infty$ such that
$$\|TP_n -P_n TP_n \|\to 0,  \quad \mbox{ as } n\to \infty.$$
In this Section, we show that every Bishop operator $T_{\alpha}$ is indeed biquasitriangular in $L^2[0,1)$, that is, both $T_{\alpha}$ and its adjoint
$T_{\alpha}^*$ are quasitriangular operators. We will derive some consequences regarding the approximation of $T_{\alpha}$.\smallskip

In order to prove the result, we will consider semi-Fredholm operators. Let $T$ be in  $\mathcal{L}(H)$ and denote by $\Ker T$ and $\Ran T $ its kernel and its range, respectively. Recall that $T$ is called \emph{semi-Fredholm} if $\Ran T $ is closed and either
the dimension of the kernel of $T$ or the dimension of the kernel of the adjoint $T^*$ is finite. In this case, the \emph{index} of $T$ is defined by
$$
{\rm index\; } T= \dim (\Ker T) - \dim (\Ker T^*)
$$
The following remarkable theorem by Douglas and Pearcy \cite{Douglas-Pearcy} and Apostol, Foias and Voiculescu \cite{AFV-1} (see also \cite[Chapter 6]{Herrero}) is the key fact relating semi-Fredholm operators to quasitriangular ones:

\begin{theorem}[Douglas, Pearcy- Apostol, Foias, Voiculescu]
An operator $T$  is quasitriangular in $H$ if and only if ${\rm index }(T-\lambda I)\geq 0$  for each complex number $\lambda \in \mathbb{C}$  such that $T-\lambda I$  is semi-Fredholm.
\end{theorem}

We are in position now the prove the following result:

\begin{theorem}\label{thm-quasitriangular}
For every irrational $\alpha \in (0,1)$, the Bishop operator $T_{\alpha}$ in $L^2[0,1)$ is biquasitriangular.
\end{theorem}

\begin{proof}
Let $\lambda \in \mathbb{C}$  such that $T_{\alpha}-\lambda I$  is semi-Fredholm. In particular, both $T_{\alpha}-\lambda I$ and its adjoint $T^*_{\alpha}-\overline{\lambda} I$ have closed range.  Since the point spectrum of $T_{\alpha}$ is empty, one has that $\lambda$ is in the resolvent of $T_{\alpha}$, that is, $T_{\alpha}-\lambda I$ is invertible. Hence, ${\rm index }(T_{\alpha}-\lambda I)=0$.
Since ${\rm index }(T^*_{\alpha}-\overline{\lambda} I)= - {\rm index }(T_{\alpha}-\lambda I)$, it follows that both $T_{\alpha}$ and its adjoint $T_{\alpha}^*$ are quasitriangular operators in $L^2[0,1)$, and the theorem
is proved.
\end{proof}

A few consequences may be derived from Theorem \ref{thm-quasitriangular} in terms of approximation of $T_{\alpha}$ by linear bounded operators. For instance, by means of \cite[Theorem 6.15]{Herrero}, one has straightforwardly that for every irrational $\alpha \in (0,1)$, the operator $T_{\alpha}$ is the norm limit of algebraic operators. Recall that an operator is called \emph{algebraic} if there exists a polynomial $p$ such that $p(T)$ is the zero operator. Clearly algebraic operators have non-trivial closed invariant subspaces. At this regard, it is worthy to point out that indeed, for every irrational $\alpha \in (0,1)$,  $T_{\alpha}$ is norm limit of nilpotent operators in $L^p[0,1)$. Namely, for any positive integer $n$, let $\phi_n(t)= t \cdot \, 1_{[1/n,1)}(t)$ for $t\in [0,1)$ and consider the Bishop-type operator $T_{\phi_n, \, \alpha}$  defined  by
$$
T_{\phi_n, \, \alpha} f(t) = \phi_n(t) f(\{t+\alpha\}), \qquad t\in [0,1),
$$
for $f\in L^p[0,1)$, $1\leq p<\infty$. Clearly, $(T_{\phi_n, \, \alpha})_{n\geq 1}$ are linear bounded operators in $L^p[0,1)$ converging in norm to $T_{\alpha}$. Moreover, having in mind that {$\tau_\alpha(t)=\{t+\alpha\}$} with $\alpha$ irrational in $[0,1)$ is an ergodic transformation in $L^p[0,1)$, one deduces that $T_{\phi_n, \, \alpha}$ is nilpotent for every $n\geq 1$.

\medskip

\begin{remark}
The fact that for every irrational $\alpha \in (0,1)$, both the Bishop operator $T_{\alpha}$ and its adjoint $T_{\alpha}^*$ in $L^2[0,1)$ are norm limit of nilpotent operators in $L^2[0,1)$ could be derived upon applying a theorem of Apostol, Foias and Voiculescu \cite{AFV-1} which states that a linear bounded operator $T$  is the norm limit of nilpotent operators if and only if it is biquasitriangular and both its spectrum and essential spectrum are connected and contain 0.
The spectrum of $T_{\alpha}$ was first studied by Parrott \cite{PARROTT_thesis} in his Ph.D. thesis, who analyzed the different parts of the spectrum and proved, in particular, that
\begin{equation}\label{spectrum of Bishop}
\sigma(T_{\alpha})=\{\lambda\in \mathbb{C}:\; |\lambda|\leq e^{-1}\}
\end{equation} for any irrational $\alpha\in (0,1)$.
Moreover, he also showed that the spectrum $\sigma(T_{\alpha})$ coincides with the essential spectrum $\sigma_e(T_{\alpha})$. Recall that if $\mathcal{K}(H)$ denotes the two-sided ideal of the compact operators in $H$, the {\em essential spectrum} of a linear bounded operator $T$ consists of the set of complex numbers $\lambda\in \mathbb{C}$ such that $T-\lambda I$ is not invertible modulo compact operators, that is, $T-\lambda I$ is not invertible in the {\em Calkin algebra} $\mathcal{L}(H)/ \mathcal{K}(H)$ (see Conway's monograph \cite{Con}, for instance, for more on the essential spectrum).
\end{remark}

\section{Bishop operators \texorpdfstring{$T_{\alpha}$}{Ta} with non-trivial invariant subspaces: enlarging the class of irrationals \texorpdfstring{$\alpha$}{a} }\label{Section 3}

In this Section, we extend the set of  known values of $\alpha$ for which the Bishop operator  $T_{\alpha}$
acting on $L^p[0,1)$, $1\leq p<\infty$, has non-trivial  closed invariant subspaces (observe that for $p=\infty$, the existence follows since $L^{\infty}[0,1)$ is not separable).

The main goal of this section will be providing a careful approach to those irrationals in order to apply Atzmon's Theorem \cite{Atzmon-1984}, by means of a functional calculus based on Beurling algebras, that is, algebras of continuous functions on the unit circle $\mathbb{T}$ with a restricted growth of the Fourier sequences. In order to consider such approach, we will consider the operator
\begin{equation}\label{normalization}
 \widetilde{T}_{\alpha}= e \; T_{\alpha}
\end{equation}
which, by means of the Spectral Theorem and equation (\ref{spectrum of Bishop}), satisfies that the spectrum $\sigma (\widetilde{T}_{\alpha})=\overline{\mathbb{D}}$.
For the sake of completeness, we recall some  results
regarding Atzmon's Theorem and Flattot's result \cite{flattot} to state the result in context. We refer to Chapter 5 in \cite{Chalendar-Partington_book} for a complete account of it.

\subsection{Beurling algebras and a theorem of Atzmon}

Given $(\rho_n)_{n\in \mathbb Z}$ a sequence in $[1, +\infty)$, let $\mathcal{A}_{\rho}$ consists of the Banach space of functions $f$  continuous in $\mathbb{T}$ such that   the norm is given by
$$
\|f\|_{\rho}=\sum_{n\in \mathbb{Z}} |\hat{f}(n)| \rho_n,
$$
where $(\hat{f}(n))_{n\in \mathbb{Z}}$ denotes the sequence of Fourier coefficients of $f$.
Observe that if $(\log \rho_n)_{n\in \mathbb{Z}}$ is sub-additive, that is, if $\rho_{m+n}\leq \rho_n\, \rho_m$ for all $n,\, m \in \mathbb{Z}$, then
$\mathcal{A}_{\rho}$ is a unital Banach algebra under pointwise multiplication. Note that the function algebra $\mathcal{A}_{\rho}$ is isometrically isomorphic to the weighted convolution algebra $\ell^1(\mathbb{Z}, (\rho_n)_n)$; commonly known as \emph{Beurling algebra}.

\begin{definition}
A sequence of real numbers $(\rho_n)_{n\in \mathbb Z}$ such that $\rho_0=1$ and $\rho_n\ge 1$ for all $n\in\mathbb Z$, is called a \emph{Beurling sequence} if
\[
 \rho_{m+n}\le \rho_m \rho_n   \qquad \forall m,n\in \mathbb Z
 \qquad\text{and}\qquad
 \sum_{n\in\mathbb Z} \frac{\log \rho_n}{1+n^2} <\infty.
\]
\end{definition}

One of the key results regarding the Banach algebra $\mathcal{A}_{\rho}$ when $\rho=(\rho_n)_{n\in \mathbb Z}$ is a Beurling sequence is that $f\in \mathcal{A}_{\rho}$ is invertible if and only if $f(e^{i\theta})\neq 0$ for all $\theta\in [0, 2\pi]$. Moreover, the Banach algebra $\mathcal{A}_{\rho}$ is regular \cite[Theorem 5.1.7]{Chalendar-Partington_book}. Recall that a  function algebra $\mathcal{A}$ on a compact space $X$ is said to be {\em regular} if for all $p\in X$ and all compact subsets $K$ of $X$ with $p\not \in K$, there exists $f\in \mathcal{A}$ such that $f(p)=1$ and $f=0$ in $K$. One advantage of regularity in a function algebra on $\mathbb{T}$  is that it enables to construct two non-zero functions whose product is identically zero; and this, combined with a functional calculus argument, gives a strategy for obtaining invariant subspaces. This was pursued by Davie \cite{DAVIE_inv_subs_bishop} and refined by MacDonald \cite{MACDONALD_inv_bishop_type} and Flattot \cite{flattot} by means of Atzmon's theorem. In order to state it, let us recall the definition of $\rho$-regular numbers:

\begin{definition}
Let $\rho=(\rho_n)_{n\in \mathbb Z}$ be a Beurling sequence. An irrational $\alpha$  is said to be $\rho$-regular if there exists $m_0\in \mathbb{N}$ and two functions $h_1$, $h_2$ satisfying
\begin{equation*}
\frac{h_1(n)\, \log \rho_n}{n\, \log n}\to \infty \; \; \mbox{ and } \; \;
\frac{h_2(n)\, \log n}{\log \rho_n}\to 0 \; \mbox{ as } n\to \infty,
\end{equation*}
such that, for all $n>n_0$, there exists $p, q\in \mathbb{N}$, with $(p,q)=1$, satisfying
$$
\left | \alpha- \frac{p}{q} \right | \leq \frac{1}{q^2} \; \mbox{ and }
\;
h_1(n)\leq q\leq h_2(n).
$$
\end{definition}
Davie \cite{DAVIE_inv_subs_bishop} made the choices $\log \rho_n= |n|^\rho$  with $\frac{1}{2}<\rho<1$, $h_1(n)=n^{\rho'}$ with $0<\rho'<\frac{1}{2}$ and
$h_2(n)=n^{1/2}$, which characterized the non-Liouville numbers. Flattot \cite[Theorem 4.6]{flattot} extended it to a larger class including some Liouville numbers, by taking
$h_1(n)= (\log n)^{2+\varepsilon}$, $h_2(n)= n/h_1(n)$
and
$\log \rho_n=h_2(n)\sqrt{h_1(n)}$. In particular, using the language of continued fractions, if
$(a_j/q_j)_{j=0}^\infty$ are the convergents of $\alpha$, the limit of his result
(see \cite[Remark 5.4]{flattot})
gives the existence of non-trivial  invariant subspaces for $T_\alpha$
when
\begin{equation}\label{eq:flattot}
 \log q_{j+1} =O\big(q_j^{1/2-\varepsilon}\big)    \qquad
 \text{for any }\varepsilon>0.
\end{equation}
Note that the condition \eqref{eq:flattot} holds for instance for the classical Liouville number $\alpha=\sum_{j=0}^{\infty} 10^{-j!}$.

As mentioned, Atzmon's Theorem \cite{Atzmon-1984} was a key result in MacDonald  and Flattot approaches.  In order to state it,  we say that a sequence $(a_n)_{n\in \mathbb Z}$ is \emph{dominated} by another sequence, both non-negative, $(b_n)_{n\in \mathbb Z}$ if $a_n\le c \, b_n$ for all $n\in\mathbb Z$ and some constant $c>0$.

\begin{theorem}[Atzmon \cite{Atzmon-1984}]\label{te:atzmon}
Let $\mathcal{X}$ be a Banach space and $T$ a linear bounded operator in $\mathcal{X}$. Suppose that there exist sequences $(x_n)_{n\in \mathbb{Z}}$ in $\mathcal{X}$ and $(y_n)_{n\in \mathbb{Z}}$ in $\mathcal{X}^*$ with $x_0\neq 0$, $y_0\neq 0$ such that
$$
T x_n= x_{n+1} \; \mbox{ and } \; T^* y_n= y_{n+1}
$$
for all $n\in \mathbb{Z}$. Suppose further that both sequences  $(\|x_n\|)_{n\in \mathbb{Z}}$ and  $(\|y_n\|)_{n\in \mathbb{Z}}$ are dominated by Beurling sequences, and there are at least $\lambda \in \mathbb{T}$ at which the following vector-valued functions $G_x$ and $G_y$ defined on $\mathbb{C}\setminus \mathbb{T}$ do not both possess analytic continuation into a neighborhood of $\lambda$:
\begin{equation}\label{G_x}
G_x(z)=\left \{ \begin {array}{ll}
\sum_{n=1}^{\infty} x_{-n} z^{n-1} & \mbox{ if } |z|<1,\\
\noalign{\medskip}
- \sum_{n=-\infty}^{0} x_{-n} z^{n-1} & \mbox{ if } |z|>1,
\end{array}
\right.
\end{equation}
\begin{equation}\label{G_y}
G_y(z)=\left \{ \begin {array}{ll}
\sum_{n=1}^{\infty} y_{-n} z^{n-1} & \mbox{ if } |z|<1,\\
\noalign{\medskip}
- \sum_{n=-\infty}^{0} y_{-n} z^{n-1} & \mbox{ if } |z|>1,
\end{array}
\right.
\end{equation}
Then either $T$ is a multiple of the identity or it has a non-trivial hyperinvariant subspace.
\end{theorem}

A few remarks are in order.
Here $T^{-1}x=w$ means $Tw=x$ even if $T$ is not invertible, and in this way
$\big(\|T^n x\|\big)_{n\in\mathbb Z}$
means a sequence
$\big(\|x_n\| \big)_{n\in\mathbb Z}$
with $x_0=x$ and $x_{n+1}=Tx_n$.

Observe also that the fact that $(\|x_n\|)_{n\in \mathbb{Z}}$ and  $(\|y_n\|)_{n\in \mathbb{Z}}$ are dominated by Beurling sequences ensures that the Laurent series defining $G_x$ and $G_y$ converge absolutely in $\mathbb{C}\setminus \mathbb{T}$.
In addition, both $G_x$ and $G_y$ are analytic functions in $\mathbb{C}\setminus \mathbb{T}$ and at $\infty$, and hence, by Liouville's Theorem, each must have at least one singularity on the unit circle. At this regard, in order to apply Atzmon's Theorem to $T_{\alpha}$, as observed by MacDonald (see \cite[Claim pp. 307]{MACDONALD_inv_bishop_type}) one has the following:

\begin{proposition} \label{macdonald claim}
Let $\alpha \in (0,1)$ be an irrational number and $T_{\alpha}$ the Bishop operator acting on $L^p[0,1)$. Let $x_0\in L^p[0,1)$ and consider $G_{x_0}$ and $G_{e^{2\pi \imath t} x_0}$ the analytic functions in $\mathbb{C}\setminus \mathbb{T}$ given by equation (\ref{G_x}) associated to $x_0$ and $e^{2\pi \imath t} x_0$, respectively.
Then $z_0 \in \mathbb{T}$ is a singularity of $G_{x_0}$ if and only if $e^{-2\pi \imath \alpha} z_0$ is a singularity of $G_{e^{2\pi \imath t} x_0}$.
\end{proposition}

The proof is just a consequence of the fact that $T_{\alpha}$ is similar to $e^{2\pi \imath \alpha}T_{\alpha}$ via the bilateral shift operator $W=M_{e^{2\pi \imath t}}$ in $L^p[0,1)$ (and unitary equivalent in $L^2[0,1)$).

With Proposition \ref{macdonald claim} at hand, one deduces that $T_{\alpha}$ has non-trivial hyperinvariant subspaces in $L^p[0,1)$, $1\leq p<\infty$, by means of Atzmon's Theorem  as far as there exist $x, y\in L^p[0,1)\setminus\{0\}$ such that
$\big(\|\widetilde{T}_\alpha^n x\|\big)_{n\in\mathbb Z}$ and $\big(\|(\widetilde{T}_\alpha^*)^n y\|\big)_{n\in\mathbb Z}$ are dominated by Beurling sequences. We state it  for later reference.

\begin{theorem}\label{te:atzmon_macdonald}
Given $\alpha \in (0,1)$ be an irrational number, if
there exist $x,y\in L^p[0,1)\setminus\{0\}$, $1\leq p<\infty$, such that
$\big(\|\widetilde{T}_\alpha^n x\|\big)_{n\in\mathbb Z}$ and $\big(\|(\widetilde{T}_\alpha^*)^n y\|\big)_{n\in\mathbb Z}$ are dominated by Beurling sequences then  $T_{\alpha}$ has a non-trivial  hyperinvariant closed subspace.
\end{theorem}

We are now in position to state the main result of this section:

\begin{theorem}\label{te:main_result}
Let $\alpha \in (0,1)$ an irrational number
and $(a_j/q_j)_{j=0}^\infty$ the convergents in its continuous fraction.
If
\begin{equation}\label{eq:denominators_growth}
 \log q_{j+1} =O\left(
 \frac{q_j}{(\log q_j)^3}
 \right)
\end{equation}
then $T_{\alpha}$ has a non-trivial  closed hyperinvariant subspace in $L^p[0,1)$, for $1\leq p\leq \infty$.
\end{theorem}

Observe that Theorem \ref{te:main_result} relaxes the condition provided by Flattot \eqref{eq:flattot},  allowing the exponent $1$ instead of $1/2$ and quantifying the role of $\varepsilon$. As we shall establish in Section \ref{Section 4}, Theorem~\ref{te:main_result} is essentially the best possible result attainable from Theorem~\ref{te:atzmon} and any improvement beyond the power of $\log q_j$ seems to require different functional analytical results.

Before proving Theorem \ref{te:main_result}, we consider a short derivation of the results of Davie and Flattot from Theorem~\ref{te:atzmon_macdonald} which highlights
arithmetical considerations encapsulated in the Banach algebra arguments and  may give some insight into the problem. In particular, it constitutes a  simplification of the Theorem in \cite{flattot}.

We will see that the aforementioned results from \cite{DAVIE_inv_subs_bishop} and \cite{flattot} follow choosing in Theorem~\ref{te:atzmon_macdonald}
\[
 x=y=1_{\mathcal{B}_{\alpha}}
 \quad\text{with}\quad
 \mathcal{B}_{\alpha}
 =
 \Big\{
 \frac{1}{20}
 <t<
 \frac{19}{20}
 \,:\,
 \langle t-n\alpha\rangle >\frac{1}{20n^2}
 ,\
 \forall n\in\Z^*
 \Big\}
\]
where
$\langle x\rangle =\min\big(\{x\}, 1-\{x\}\big)$ is the distance to the closest integer
and $\Z^* =\Z\setminus\{0\}$.
As a matter of fact $\mathcal{B}_\alpha$ is none other than a variant of the sets $E_t$ appearing in those papers.
We point out that replacing in the definition of $\mathcal{B}_\alpha$ the condition by $\langle q_j t\rangle>Cq_j^{-1}$, with $C$ {a} certain constant, would give a more manageable set but we prefer not to proceed in this way to keep the analogy with \cite{DAVIE_inv_subs_bishop} and \cite{flattot}.

Observe also that $\mathcal{B}_{\alpha}$ has positive measure and hence $x$ and $y$ do not vanish identically as elements of $L^p[0,1)$. Note that $\langle t\rangle \le \delta$ defines in $[0,1)$ a set of measure  $2\delta$ for $\delta<1/2$. Then the measure of the complement of $\mathcal{B}_\alpha$ in $[0,1)$ is at most one twentieth of
$2+2\sum_{n\in\Z^*} n^{-2}=2+2\pi^2/3<20$  and consequently $\mathcal{B}_\alpha$ has positive measure.

In what follows, if $\alpha$ is an irrational number, and $(q_j)_{j=0}^\infty$ denotes the denominators of its convergents, an important fact we are going to use about continued fractions is that   $(q_j)_{j=0}^\infty$ is an increasing sequence of positive integers such that \cite[\S7.5]{MiTa}
\begin{equation}\label{eq:cont_frac}
 (2Q)^{-1}<
 \langle
 q\alpha
 \rangle
 <Q^{-1}
 \qquad\text{for}\quad
 q = q_j,\quad Q=q_{j+1}.
\end{equation}
This is more precise than Dirichlet's theorem, which assures $\langle q \alpha\rangle <q^{-1}$ for infinitely many values of~$q$.
From here on out, we use $q$ and $Q$ to indicate consecutive terms of $(q_j)_{j=0}^\infty$ as in \eqref{eq:cont_frac}.

\subsection{The results of Davie and Flattot}\label{sec:davie_bound} In this subsection, we derive the results of Davie and Flattot, providing a simplification of the Theorem in \cite{flattot}.

In the sequel, the real function
\begin{equation}\label{eq:Ln}
 L_n(t)
 =
 \sum_{j=0}^{n-1}\big(1+\log \{t+j\alpha\}\big)
\end{equation}
plays a fundamental role because it is plain to check
\begin{equation}\label{eq:TL}
 \begin{cases}
 \begin{array}{ll}
 \widetilde{T}_{\alpha}^n f(t)=e^{L_n(t)} f(\{t+n\alpha\}),
 &
 \widetilde{T}_{\alpha}^{-n} f(\{t+n\alpha\})=e^{-L_n(t)} f(t),
 \\[7pt]
 (\widetilde{T}_{\alpha}^*)^{n} f(\{t+n\alpha\})=e^{L_n(t)} f(t),
 &
 (\widetilde{T}_{\alpha}^*)^{-n} f(t)=e^{-L_n(t)} f(\{t+n\alpha\})
 \end{array}
 \end{cases}
\end{equation}
for $n\in\Z^+$ and also for $n=0$ defining $L_0=0$.

\medskip

For latter reference it is convenient to manipulate a little the definition of $L_n(t)$ when $n=q$.
\begin{lemma}\label{le:Lq}
 For  $q$ as in \eqref{eq:cont_frac} fixed there exist $1/2<|\delta|<1$ and $|\delta_\ell|<1$ with the same sign such that for any
 $k\in\Z$
 \[
  L_q(t+kq\alpha)
  =
  \sum_{\ell=0}^{q-1}
  \bigg(1+\log \big\{t+\frac{\ell}{q}+\frac{k\delta+\delta_\ell}{Q}\big\}\bigg).
 \]
\end{lemma}
\begin{proof}
 By \eqref{eq:cont_frac}, we can write $\alpha=a/q+\delta/(qQ)$ where $a/q$ is a convergent of $\alpha$. Then the fractional part in \eqref{eq:Ln} is
 $\big\{t+{ja}/{q}+{k\delta}/{Q}+{j\delta}/({qQ})\big\}$.
 The map $j\mapsto aj$ is invertible modulo~$q$. If $\ell \mapsto j_\ell$ is its inverse with $0\le j_\ell<q$, the result follows taking $\delta_\ell = j_\ell \delta/q$.
\end{proof}

The following estimates for $L_n$
are variations on those for $F_m$ in \cite{DAVIE_inv_subs_bishop}.
\begin{lemma}\label{le:davie_str}
 There exists an absolute  constant $C>0$ such that for $n\in\Z^+$
 \[
  L_n(t)
  \le
  C
  \Big(
  r+\frac{n}{q}\log(q+1)
  \Big)
  \qquad
  \text{for every}\quad t\in\R
 \]
 where $r$ is the remainder when $n$ is divided by $q$. Moreover we have
 \[
  L_n(t)
  \ge
  -C
  \Big(
  r'+\frac{n+q}{q}\log(\mu^{-1}+q)
  \Big)
  \qquad
   \text{if}\quad
  \min_{0\le j<r'+n}
  \{t+j\alpha\}
  \ge
  \mu>0
 \]
 where $r'=0$ if $r= 0$ and $r'=q-r$ otherwise.
\end{lemma}
\begin{proof}
 Separating the last $r$ terms in $L_n$, we have
 \[
  L_n(t)
  \le
  r+
  \sum_{j=0}^{n-r-1}
  \big(1+\log \{t+j\alpha\}\big)
  =
  r+
  \sum_{k=0}^{\lfloor n/q\rfloor -1}
  L_q(t+kq\alpha).
 \]
 Applying Lemma~\ref{le:Lq}, as $\delta_\ell$ has constant sign and $Q>q$,
 {on each interval $[\ell/q, (\ell+1)/q]$ there is exactly one value $\ell/q+{\delta_\ell}/{Q}$ and then we have}
 \[
  L_q(t+kq\alpha)
  \le
  q+
  \sum_{\ell=2}^{q-1}
  \log\frac{\ell}{q}
  \le C\log(q+1),
 \]
 using Stirling's approximation, which proves the first inequality.

 For the second, we expand the sum to the first multiple of $q$ not less than~$n$. Then
 \[
  L_n(t)
  =
  -
  \sum_{j=n}^{n+r'-1}
  +
  \sum_{j=0}^{n+r'-1}
  \ge
  -r'
  +
  \sum_{k=0}^{\lceil n/q\rceil -1}
  L_q(t+kq\alpha).
 \]
 As the values of $\ell/q+{\delta_\ell}/{Q}$ are confined to disjoint intervals of length $q^{-1}$,  at most two of the fractional parts in $L_q$ could nearly coincide and the smallest fractional part appearing in $L_q(t+kq\alpha)$
 is the minimum indicated in our hypothesis. Then
 \[
  L_q(t+kq\alpha)
  \ge
  -2\log\big(\mu^{-1}\big)
  +q
  +
  \sum_{\ell=1}^{q-2}
  \log\frac{\ell}{q}
  \ge -C\log(\mu^{-1}+q)
 \]
 and the result follows.
\end{proof}

\begin{corollary}\label{co:davie_bound}
Let $x=y=1_{\mathcal{B}_{\alpha}}$. Then for $n\in\Z$
\[
 \log\big(
 1+
 \|\widetilde{T}_{\alpha}^n x\|
 +
 \|(\widetilde{T}_{\alpha}^*)^{n}y\|
 \big)
 \ll
 q+\frac{|n|+q}{q}\log(|n|+q+1).
\]
\end{corollary}
\begin{proof}
 The result is trivial for $n=0$ and it follows immediately from the first part of Lemma~\ref{le:davie_str} via \eqref{eq:TL}  if $n>0$. On the other hand, the second part gives the expected bound for
 $\log\big(1+\|\widetilde{T}_{\alpha}^{-n} x\|)$ with $n\in\Z^+$
 because for $t\in\mathcal{B}_\alpha$, we can take  $\mu^{-1}=20n^2$.
 The same works for
 $\log\big(1+\|(\widetilde{T}_{\alpha}^*)^{-n} y\|)$ with $n\in\Z^+$
 because $\{t+n\alpha\}\in\mathcal{B}_\alpha$ implies
 $\langle t+n\alpha-\ell \alpha\rangle >1/(20\ell)^2$ for $-q< \ell\le n$, $\ell\ne 0$ and then $\mu^{-1}=20(n+q)^2$ is a valid choice.
\end{proof}

With this bound we can easily derive the best known result from Theorem~\ref{te:atzmon_macdonald}.

\begin{corollary}\label{co:flattot}
Let $\alpha \in (0,1)$ be an irrational number such that the convergents
$(a_j/q_j)_{j=0}^\infty$ in its continuous fraction satisfy \eqref{eq:flattot}. Then $T_\alpha$ has a non-trivial  closed hyperinvariant subspace.
\end{corollary}

\begin{proof}
 The sequence $(\rho_n)_{n\in\Z}$
 given by $\log \rho_n = C_\sigma
  {|n|}{\log^{-\sigma}(2+|n|)}$
 is clearly a Beurling sequence for any $\sigma>1$ and $C_\sigma>0$.
 By Corollary~\ref{co:davie_bound} and Theorem~\ref{te:atzmon_macdonald}, it is enough to show that for $|n|$ large we can always find $q$ such that
 \[
  q+\frac{|n|+q}{q}\log(|n|+q+1)
  =
  O
  \left(
  \frac{|n|}{\log^\sigma(2+|n|)}
  \right)
  \qquad\text{for some }\sigma>1.
 \]
 Take $q$ such that $q\le |n|^{2/3}<Q$.
 By \eqref{eq:flattot} we have
 $\log Q=O\big(q^{1/2-\varepsilon}\big)$, hence
 $q\gg \big(\log |n|\big)^{1+\sigma}$
 for $1+\sigma=(1/2-\varepsilon)^{-1}$ and the expected bound follows.

\end{proof}

\subsection{Proof of Theorem \ref{te:main_result}}\label{sec:main}

In this subsection, we address the proof of Theorem \ref{te:main_result}.\medskip

Firstly, for $t\in \mathcal{B}_\alpha$ we can take $\mu^{-1}=20(n+q)^2$ in Lemma~\ref{le:davie_str} and
if $n$ is very large in comparison with $q$
there is an asymmetry in the bounds obtained in this lemma
being the upper bound stronger.
This is reasonable since in \eqref{eq:Ln} a fractional part can be very small but not large since it is bounded by 1. Anyway, we shall see that it is possible to partially recover the symmetry getting a non biased bound for $|L_n(t)|$
by a more careful analysis than the one in \textrm{\S}\ref{sec:davie_bound}. The {improvement} is achieved when $n$ is very large in comparison to $q$, in such a way that $\log(|n|+q+1)$ is not comparable to $\log(q+1)$ in Corollary~\ref{co:davie_bound}, but it is controlled by $Q$ (see Proposition~\ref{pr:bound_n_general} below).

\begin{lemma}\label{le:bound_n_small}
If $q\mid n$, $1\le n\le Q/(100q)$ and $t_0\in \mathcal{B}_\alpha$ then
\[
 |L_n(t)|\ll \frac{n}{q}\log (q+1)
 \qquad\text{for every $t\in [0,1)$ with }|t-t_0| \le \frac{1}{(10q)^2}.
\]
\end{lemma}

\begin{proof}
 We can write  $L_n(t)=\sum_{k=0}^{n/q-1}L_q(t+kq\alpha)$.
 It is enough to prove
 $\min_{0\le j<q}\{t+kq\alpha+j\alpha\}>Cq^{-2}$ with $C$ some constant
 because in this case Lemma~\ref{le:davie_str} assures that each term in the sum contributes $O\big(\log(q+1)\big)$.

 By \eqref{eq:cont_frac}, $\alpha=a/q+\eta/(qQ)$ with $|\eta|<1$. Then
 \[
  \{t+kq\alpha+j\alpha\}
  \ge
  \big\{t_0+j\alpha+\frac{\eta k}{Q}\big\} - \frac{1}{(10q)^2}
  \ge
  \big\langle t_0+j\alpha+\frac{\eta k}{Q}\big\rangle - \frac{1}{(10q)^2}
  >
  \frac{1}{20 j^2}
  -
  \frac{2}{100q^2}
 \]
 where we have used $t_0\in\mathcal{B}_\alpha$ and $k<Q/(100q^2)$ for the last inequality. This is greater than $Cq^{-2}$ when $1\le j<q$. A similar argument applies for $j=0$ using that $\langle t_0\rangle >1/20$ for $t_0\in\mathcal{B}_\alpha$.
\end{proof}

\begin{lemma}\label{le:bound_n_large}
For $n_1,n_2\in\Z_{\ge 0}$,
if $q\mid n_2-n_1$ and $Q/(100q)\le n_2-n_1\le Q-q$ then
\[
 L_{n_2}(t)-L_{n_1}(t)\ll
 \log n_2+\frac{n_2-n_1}{q}\log (q+1)
 \qquad\text{for every }t\in \mathcal{B}_\alpha.
\]
\end{lemma}
\begin{proof}
 We start writing
 \[
  L_{n_2}(t)-L_{n_1}(t)
  =
  \sum_{k=0}^{K-1}
  L_q(t+n_1\alpha+kq\alpha)
  \qquad\text{with}\quad
  K=\frac{n_2-n_1}{q}.
 \]
 Let us call $\mu$ to the minimum of the fractional parts appearing in these terms.
 We have $\mu>1/(20n_2^2)$ because $t\in\mathcal{B}_\alpha$.

 By Lemma~\ref{le:Lq} and doing a translation $\ell\mapsto \ell+\ell_0$ modulo $q$  if the minimum is reached for a certain $k=k_0$ and $\ell=\ell_0$,  this can be expanded as
 \[
  L_{n_2}(t)-L_{n_1}(t)
  =
  \sum_{k=0}^{K-1}
  \sum_{\ell=0}^{q-1}
  \bigg(1+\log \big\{\mu+\frac{\ell}{q}+\frac{(k-k_0)\delta+\delta_\ell}{Q}\big\}\bigg).
 \]
 Note that we have employed $\delta_{\ell+\ell_0}-\delta_{\ell_0}=\delta_{\ell}$.
 We know that $K\le Q/q-1$ and recalling the properties of $\delta$ and $\delta_\ell$ in  Lemma~\ref{le:Lq}, we have
 \[
  \frac{{|k-k_0|}}{2Q}
  <
  \Big|
  \frac{{(k-k_0)}\delta+\delta_\ell}{Q}
  \Big|
  <\frac{K}{Q}
  \le \frac 1q-\frac 1Q.
 \]
 If $\ell\ne 0$ then the fractional part can be safely compared with that of $\ell/q$
 to get $O\big(\log (q+1)\big)$ for the sum on $\ell\ne0$ and fixed each value of $k$. This gives~$O\big(K\log (q+1)\big)$.
 The contribution of $\ell=0$ is comparable to
 \[
  \log(\mu^{-1})
  +
  K
  +
  \Big|
  \sum_{k=1}^{K}
  \log \big( \frac{k}{Q}\big)
  \Big|
  \ll
  \log n_2 +K+
  \Big|{\sum_{k=1}^K}\log\frac{K}{Q}\Big|
  +
  \Big|\sum_{k=1}^K\log\frac{k}{K}\Big|.
 \]
 The {last} sum is $O(K)$ by Stirling's approximation. This gives the expected bound noting $Q/K\ll q^2$.
\end{proof}

With these lemmas we are ready to get an improvement of Lemma~\ref{le:davie_str} for restricted values of $t$.

\begin{proposition}\label{pr:bound_n_general}
Assume $Q\ge 4(10 q)^4$, $1\le n\le Q^{3/2}$  and let $N$ be the closest multiple of $Q$ to $n$. Then for
$t\in {\mathcal{B}_{\alpha}}$ we have
\[
 L_n(t)
 \ll
 q+\frac{|n-N|}{q}\log (q+1)+\frac{n+Q}Q \log(n+1).
\]
\end{proposition}
\begin{proof}
 We introduce the decomposition
 \begin{equation}\label{eq:dec}
  L_n=L_N+\big(L_n-L_{n'}\big)+\big(L_{n'}-L_N\big)
 \end{equation}
 where $n'=N\pm m$, $m\in\Z^+$, with $\pm m$ the closest multiple of $q$ to $n-N$ (here the $\pm$ indicates the sign of $n-N$).
 Clearly we have
 $0\le m\le |n-N|+q/2$.

 Applying Lemma~\ref{le:davie_str} with $Q$ instead of $q$ and $\mu^{-1}=20N^2$, we have
 \[
  L_N(t)
  \ll
  \frac{N}{Q}\log(N+1)
  \ll
  \frac{n}{Q}\log(n+1).
 \]

 If $n>n'$, $L_n(t)-L_{n'}(t)$ is $L_{n-n'}(t+n'\alpha)$
 and if $n<n'$ is  $-L_{n'-n}(t+n\alpha)$. As $|n-n'|<q$ in both cases
 Lemma~\ref{le:davie_str}
 with $\mu^{-1}=20(n+2q)^2$
 assures
 \[
  L_n(t)-L_{n'}(t)
  \ll
  q+\log(n+q)\ll q+\log n.
 \]

 Finally, we have to deal with the last term in \eqref{eq:dec}.
 If $Q/(100q)<m$ then we are under the hypotheses of Lemma~\ref{le:bound_n_large} that gives
 \[
  L_{n'}(t)-L_{N}(t)
  \ll
  \log (N+m)+\frac{m}{q}\log (q+1).
 \]
 Hence
 \[
  L_{n'}(t)-L_{N}(t)
  \ll
  \log(n+1)+\frac{|n-N|}{q}\log (q+1).
 \]
 If $m\le Q/(100q)$, note firstly
  \[
  \langle N\alpha\rangle
  \le
  \frac{N}{Q} \langle Q\alpha\rangle
  \le
  \frac{N}{Q^2}
  \le
  \frac{n+Q/2}{Q^2}
  \le
  \frac{1}{(10q)^2}.
 \]
 If $n'\ge N$
 we write
 $L_{n'}(t)-L_N(t)=L_m(t+N\alpha)$
 and the previous bound proves that we can apply Lemma~\ref{le:bound_n_small} to get $O\big(q^{-1}m\log(q+1)\big)$.
 If $n'< N$ then
 $L_{N}(t)-L_{n'}(t)$
 coincides with $L_m(t+N\alpha)$ formally changing $\alpha$ by $-\alpha$ in the definition of $L_m$. As the denominators of the convergents of $\alpha$ and $-\alpha$ coincide except for a unit shift in the indexes, the same argument applies.

 Adding the contribution of the three terms in \eqref{eq:dec} we get the result.
\end{proof}

The analogue of Corollary~\ref{co:davie_bound} is:
\begin{corollary}\label{co:norm_bound}
Let $x=y=1_{{\mathcal{B}_{\alpha}}}$.  For any $|n| \le Q^{3/2}$ we have
\[
 \log\big(1+
 \|\widetilde{T}_{\alpha}^n x\|
 +
 \|(\widetilde{T}_{\alpha}^*)^{n}y\|
 \big)
 \ll
 q
 +\frac{|n|}{q}\log(q+1)
 +\frac{|n|+Q}{Q}\log(|n|+2).
\]
\end{corollary}
\begin{proof}
 We are going to show that the bound holds  for
 $A_n=\log \big(1+ \|\widetilde{T}_{\alpha}^n x\|+\|(\widetilde{T}_{\alpha}^*)^{n}y\|)$,
 $B_n=\log \big(1+ \|\widetilde{T}_{\alpha}^{-n} x\|)$
 and
 $C_n=\log \big(1+ \|(\widetilde{T}_{\alpha}^*)^{-n}y\|)$
 with $n\in\Z^+$.

 For $A_n$, it follows substituting in \eqref{eq:TL} the first bound of Lemma~\ref{le:davie_str}.

 If $Q< 4(10q)^4$ then $\log(|n|+q+1)\ll\log(q+1)$ and the bound for $B_n$ and $C_n$ follows from Corollary~\ref{co:davie_bound}.

 If $Q\ge 4(10q)^4$ Proposition~\ref{pr:bound_n_general} gives the bound for $B_n$.

 It remains to bound $C_n$ if $Q\ge 4(10q)^4$. With this purpose, we rewrite the last formula in \eqref{eq:TL} as
 $(T_\alpha^*)^{-n} f\big(\{t-n\alpha\}\big)= e^{-L_n(t-n\alpha)}f(t)$ and we note
 \[
  L_n(t-n\alpha)
  =
  -\log\{t\}
  +\log\{t-n\alpha\}
  +\sum_{j=0}^{n-1}
  \big(
  1+\log\{t-j\alpha\}
  \big).
 \]
 The sum coincides with $L_n(t)$ replacing $\alpha$ by $-\alpha$. As we mentioned before, the convergents of $\alpha$ and $-\alpha$ have the same denominators and then Proposition~\ref{pr:bound_n_general} applies also for this sum. On the other hand,
 $\log\{t\}$ and $\log\{t-n\alpha\}$ are $O\big(\log(|n|+1)\big)$ if $t\in\mathcal{B}_\alpha$.
\end{proof}

Once we have got this bound, the proof of our main result parallels that of Corollary~\ref{co:flattot}.

\begin{proof}[Proof of Theorem~\ref{te:main_result}]
Given $n\ne 0$, choose $q$ such that $q\le |n|^{2/3}<Q$.
In this range
\[
 q+\frac{|n|+Q}{Q}\log(|n|+2)
 \ll
 |n|^{2/3}+|n|^{1/3}\log(|n|+2)
\]
and by the condition (\ref{eq:denominators_growth}),
\[
 \frac{|n|}{q}\log(q+1)
 \ll
 \frac{|n|}{\log Q (\log \log Q)^2}\ll \frac{|n|}{\log |n| (\log\log |n|)^2}.
\]
Therefore by Corollary~\ref{co:norm_bound}, there exists $C>0$ such that for every $n\in\Z$
\begin{equation}\label{sp_bound}
 \max\big(\|\widetilde{T}_{\alpha}^n x\|,\|(\widetilde{T}_{\alpha}^*)^n y\|\big)
 \le
 \exp\left(
 \frac{C|n|}{\log(2+|n|)(\log \log (4+|n|))^2}
 \right)
\end{equation}
and the
result follows from Theorem~\ref{te:atzmon_macdonald} because the right hand side is a Beurling sequence.
\end{proof}

\begin{remark} Note that for Bishop-type operators of the form $T_{s, \alpha} f(t) = t^s  f(\{t+\alpha\}) $ where $s > 0$,  all the bounds computed above remain true replacing $L_n(t)$ by
		\begin{equation*}
			L_{s,n}(t) = \sum_{j=0}^{n-1} \big(s + s\log\{t+j\alpha\}\big),
		\end{equation*}
		and considering again $x=y=1_{\mathcal{B}_\alpha}$. This clearly follows from the fact $L_{s,n}(t)=s L_n(t)$. Therefore, Theorem \ref{te:main_result} is also valid for every $T_{s, \alpha}$ with $s>0$; and, in particular, we obtain a generalization of \cite[Theorem 4.7]{flattot}.
\end{remark}

\section{The limits of Atzmon Theorem}\label{Section 4}

In this section we shall show that it is not possible to improve much on Theorem \ref{te:main_result} by applying Atzmon's Theorem (Theorem \ref{te:atzmon}) to $\widetilde T_{\alpha}$. Before stating the main result of the section, observe that if $L^0[0,1)$ denotes the space of (classes of) measurable functions defined almost everywhere on [0,1), $\widetilde T_{\alpha}$ is a bijection in $L^0[0,1)$ with inverse:
$$
\widetilde T_{\alpha}^{-1}f(t) = e^{-1} \frac{f(\{t -\alpha\})}{\{t-\alpha\}}, \qquad t\in [0,1).
$$
Nevertheless, in $L^p[0,1)$, $1\leq p<\infty$, the operator $\widetilde T_{\alpha}$ is an injective, dense range operator. Hence, there exists a dense set of functions $g\in L^p[0,1)$ which have an infinite chain of backward iterates, that is, for all $n>0$ there is $g_n\in L^p[0,1)$, unique,  such that $\widetilde T_{\alpha}^{n} g_n=g$ (see \cite[Corollary 1.B.3]{Beauzamy}, for instance). As an abuse of notation in the next theorem, for $f\in L^p[0,1)$ and $n>0$, we will denote by $\|\widetilde T_{\alpha}^{-n} f\|$ the norm of the $n$-th backward iterate $\widetilde T_{\alpha}^{-n} f$ whenever it belongs to $L^p[0,1)$ or $\infty$, otherwise.
Our main aim in this section is to prove the following:

\begin{theorem}\label{te:atzmon_limit}
Let us define $\mathcal{M}$ as the set of irrationals such that the convergents $(a_j/q_j)_{j=0}^\infty$  in its continuous fraction satisfy
\begin{equation*}
 \log q_{j+1} =O\left(
 \frac{q_j}{\log q_j}
 \right)
\end{equation*}
for every $j\ge 0$. Then, if $\alpha$ is an irrational not in $\mathcal{M}$ we have
\[
 \sum_{n=-\infty}^{+\infty} \frac{\log (1+\|\widetilde T_{\alpha}^n f\|)}{1+n^2} =+\infty
\]
for any non-zero $f\in L^p[0,1)$, for $1\leq p< \infty$.
\end{theorem}

Note that this result shows that there does not exist a sequence $(x_n)_{n\in \mathbb Z}$ satisfying the requirements in the statement of Theorem \ref{te:atzmon} whenever $\alpha \in \mathcal{M}$, and hence establishes a threshold limit in the growth of the denominators of the convergents of $\alpha$ for the application of Atzmon's Theorem to Bishop operators.\smallskip

In order to prove Theorem \ref{te:atzmon_limit}, we will show that either $\|\widetilde T_{\alpha}^n f\|$ or $\|\widetilde T_{\alpha}^{-n} f\|$ is large for many values of $n$. To accomplish such a task, we consider the equation
\begin{equation}\label{eq:T_plus_inverse}
 \|\widetilde T_{\alpha}^n f\|^p+\|\widetilde T_{\alpha}^{-n}f\|^p=\int_0^1 ( e^{pL_n(t-n\alpha)} + e^{-pL_n(t)} ) |f(t)|^p \, dt
\end{equation}
for any $n\ge 1$, which follows directly from \eqref{eq:Ln}, \eqref{eq:TL} and a change of variable. Now, $\alpha\not\in \mathcal{M}$ means that it is very well approximable by some rationals $a/q$,
which will imply that $L_n(t-n\alpha)$ is essentially identical to $L_n(t-n\frac aq)=L_n(t)$ for any $n$ near $q$ and divisible by it. In this situation, it appears that the integral in \eqref{eq:T_plus_inverse} must be large unless $|L_n(t)|$ is small, which should happen rarely. That is the basic idea behind the following result.

\begin{lemma}\label{le:sublevel_estimate}
Let $a/q$ and $A/Q$ be two consecutive convergents of $\alpha$ an irrational number, $q\ge 2$. For any $\varepsilon\in (0,1/4)$ there exists a set $S_{q,\varepsilon}\subset [0,1)$ of measure at most $20\varepsilon$ such that
\[
 \min(|L_n(t-n\alpha)|,|L_n(t)|)>\varepsilon\frac{n}{q}\log q
\]
for every $t\not\in S_{q,\varepsilon}$ and every $n\in [\varepsilon^{-2} q^2\log q,\varepsilon^2 Q/q]$.
\end{lemma}
\begin{proof}
Given $\varepsilon \in (0,1/4)$, pick any $n \in [\varepsilon^{-2}q^2 \log q, \varepsilon^2 Q/q]$. By \eqref{eq:cont_frac} we have $\alpha=a/q+\delta/(qQ)$ with $|\delta|< 1$ and our hypothesis assures $| j\delta/(qQ) |<\varepsilon^2/q^2$ for every $|j|\le n$. Hence for  $\langle q t\rangle >2\varepsilon$, we have $\langle t+ja/q\rangle>2\varepsilon/q$ and
\[
\Big|
\log\big\{t+j\alpha\big\}
-
\log\big\{t+\frac{ja}{q}\big\}
\Big|
\le
\Big|
\log\Big(\frac{2\varepsilon}{q}-\frac{\varepsilon^2}{q^2}\Big)
-
\log\Big(\frac{2\varepsilon}{q}\Big)
\Big|
\le
\frac{\varepsilon}{q}.
\]
With this and the $q$-periodicity in $j$ of $\log\{t+ja/q\}$ we deduce
\[
\Big|
L_n(t)-\frac{n'}{q} L(\{qt\})
\Big|
\le
\varepsilon\frac{n}{q}
+
\Big|
\sum_{j =n'}^{n-1}
\Big(1+\log\{t+\frac{j a}{q}\}\Big)
\Big|
\]
where
$L(x)=\sum_{\ell=0}^{q-1}\big(1+\log((x+\ell)/q)\big)$ and $n'=q \lfloor n/q\rfloor$.
The trivial bound for the last term is $q\big(1-\log(2\varepsilon/q)\big)$ which is less than $2\varepsilon n/q$ in our range.
A similar argument applies for $L_n(t-n\alpha)$. Hence
\begin{equation}\label{eq:minaux}
\min\big(
|L_n(t-n\alpha)|, |L_n(t)|
\big)
>\frac{n-q+1}{q}\,
|L(\{qt\})|
-3\varepsilon\frac nq
\qquad
\text{for}\quad
\langle q t\rangle >2\varepsilon.
\end{equation}
The function $L$ is increasing in $(0,1)$ and $L'(x)\ge \log q$. Then the measure of
$\{x\,:\, |L(x)|\le 8\varepsilon\log q\}$ is at most
$16\varepsilon$ and \eqref{eq:minaux} gives the expected bound except in the set
\[
S_{q,\varepsilon}=
\{t\in [0,1):
\langle qt \rangle \le 2\varepsilon\} \cup
\{t \in [0,1) : |L(\{qt\})|\le 8\varepsilon\log q\}
\]
which has measure at most $20\varepsilon$.
\end{proof}

With this lemma, we are ready to prove the theorem.

\begin{proof}[Proof of Theorem \ref{te:atzmon_limit}]
Without loss of generality, assume that $f\in L^p[0,1)$ has an infinite chain of backward iterates $\widetilde T_\alpha^{-n}f\in L^p[0,1)$ and suppose $\|f\|=1$.  If $\alpha$ is an irrational not in $\mathcal{M}$, we have $\limsup_{j\to\infty} \frac{\log q_{j+1}}{q_j/\log q_j}=+\infty$, there exists a subsequence $(q_{j_m})_{m\in \N}$ such that
\[
 \frac{\log Q_{j_m}}{q_{j_m}/\log q_{j_m}} >m^2,   \qquad \text{with } \quad Q_{j_m}=q_{j_m+1}
\]
for every $m > 2$. Now, consider the sets $S_{m_*}=\cup_{m\ge m_*} S_{q_{j_m},1/m^2}$, with $S_{q,\varepsilon}$ defined as in Lemma \ref{le:sublevel_estimate}. Since $\sum_{m=1}^{\infty} m^{-2}<\infty$ we have that $\lim_{m_*\to\infty} \int_{S_{m_*}} |f(t)|^p \,dt=0$, so there exists $m_*$ such that $\int_{S_{m_*}} |f(t)|^p dt <1/2$. This and \eqref{eq:T_plus_inverse} imply that
\[
  \|\widetilde T_{\alpha}^n f\|^p+\|\widetilde T_{\alpha}^{-n}f\|^p\ge \frac 12 \inf_{t\not\in S_{m_*}} \big(e^{pL_n(t-n\alpha)}+e^{-pL_n(t)}\big).
\]
By Lemma \ref{le:sublevel_estimate} with $q=q_{j_m}$, $m\ge m_*$, and $\varepsilon=1/m^2$ we have
\[
\|\widetilde T_{\alpha}^n f\|^p+\|\widetilde T_{\alpha}^{-n}f\|^p\ge \frac 12 e^{pn\log q_{j_m}/(m^2 q_{j_m})}
\]
for any $n\in [m^2q_{j_m}^2\log q_{j_m},m^{-2}Q_{j_m}/q_{j_m}]$, so that
\begin{equation}\label{eq:indexed_sum}
 \sum_{m^2q_{j_m}^3<|n|<m^{-2}Q_{j_m}/q_{j_m}}\frac{\log (1+\|\tilde T_{\alpha}^n f\|)}{1+n^2} \gg \log\Big(\frac{Q_{j_m}}{m^4 q_{j_m}^4}\Big)\frac{\log q_{j_m}}{q_{j_m}} \frac{1}{m^2}\gg 1
\end{equation}
for any $m$ sufficiently large. As a consequence of
\[
\frac{Q_{j_m}}{m^2 q_{j_m}}  <  Q_{j_m}  = q_{j_m+1} \le q_{j_{m+1}} \le (m+1)^2 q_{j_{m+1}}^3,
\]
we observe that the intervals defined by the indexes of the sum in \eqref{eq:indexed_sum} do not overlap for different values of $m$, hence the theorem follows.
\end{proof}

\section{Spectral subspaces of Bishop operators} \label{Section 5}

In this section, we deal with \emph{local spectral subspaces} of Bishop operators, which are hyperinvariant subspaces (not necessarily closed) associated to closed {subsets} of the spectrum. While local spectral subspaces are closed for a large class of operators, those satisfying the so-called \emph{Dunford property $(C)$},  as a consequence of the estimates obtained in the previous section, our main result in this section is that all Bishop operators do not belong to such a class; and therefore there exist local spectral subspaces which are not closed.

Before going further, we recall some preliminaries regarding local spectral theory, and refer to Laursen and Neumann monograph \cite{LAURSEN-NEUMANN_book} for more on the subject.

\subsection{Local spectral theory background} Let $\mathcal{X}$ denote an arbitrary complex Banach space and $\mathcal{L}(\mathcal{X})$ the space of linear bounded operators on $\mathcal{X}$. For an open subset $U\subseteq \mathbb{C} $, let $\mathcal{H}(U,\mathcal{X})$ be the Fr\'{e}chet space of analytic functions from $U$ to $\mathcal{X}$ endowed with the topology of uniform convergence on compact subsets.

Given any $T \in \mathcal{L}(\mathcal{X})$ and $x \in \mathcal{X}$, let $\rho_T(x)$ be the local resolvent of $T$ at $x$, i.e. the set of $\lambda \in \mathbb{C}$ for which there exists an open neighborhood $U_\lambda \ni\lambda$ and an analytic function $f \in\mathcal{H}(U_\lambda, \mathcal{X})$ which fulfills the equation
\begin{equation}\label{EQ_local_resolvent}
	(T-zI)\, f(z)= x, \quad\text{for every } z \in U_\lambda.
\end{equation}
By $\sigma_T(x)$ we will denote the local spectrum of $T$ at $x$, i.e. the complementary set of the local resolvent. Of course, bearing in mind that the function $f(z) = (T-zI)^{-1}\, x$ is analytic in the whole resolvent set, we have $\sigma_T(x) \subseteq \sigma(T)$. In the sequel, we shall use the following properties concerning the local spectra of an operator:
\renewcommand{\labelenumi}{(\alph{enumi})}
\begin{enumerate}
	\item $\sigma_T(a x +b y) \subseteq \sigma_T(x)\cup \sigma_T(y)$ for every $x,\, y \in \mathcal{X}$ and $a,\, b\in\mathbb{C}$.
	\item $\sigma_T(x) \subseteq \sigma(T)$ for every $x \in \mathcal{X}$.
	\item $\sigma_T(S\, x) \subseteq \sigma_T(x)$ for every $S \in\mathcal{L}(\mathcal{X})$ which commutes with $T$.
\end{enumerate}

Whenever the solution of \normalfont{(\ref{EQ_local_resolvent})} is unique for every $\lambda \in \mathbb{C}$, we will say that $T$ satisfies the single-valued extension property (abbrev. SVEP) and we will denote by $f_x(z)$ such local resolvent function. In such a case, the local spectral radius $r_T(x)$ fulfills the equality
\begin{equation*}
	r_T(x) = \max\left\lbrace |\lambda| \, : \, \lambda \in \sigma_T(x) \right\rbrace \; \text{ for every } x \in \mathcal{X}.
\end{equation*}
Reminding the point spectrum and the compression spectrum of the Bishop operators, $\sigma_p(T_\alpha)=\emptyset$ and $\sigma_c(T_\alpha)=\emptyset$, it is somewhat direct to prove that both $T_\alpha$ and $T_\alpha^*$ have the SVEP, indeed, the same holds for every Bishop-type operator \cite[Prop. 3.6]{ARTICLE-gallardo_monsalve}. \\

Our first result regarding the local spectrum of Bishop operators by means of the estimates obtained in the previous section is the following:

\begin{theorem}\label{THM-local_spectrum}
Let $\alpha\in (0,1)$ be any irrational number, $T_{\alpha}$ the Bishop operator acting on $L^p[0,1)$, $1\leq p<\infty$,  and   $$\mathcal{B}_{\alpha}
 =
 \Big\{
 \frac{1}{20}
 <t<
 \frac{19}{20}
 \,:\,
 \langle t-n\alpha\rangle >\frac{1}{20n^2}
 ,\
 \forall n\in\Z^*
 \Big\}.$$ Then, both local {spectra} $\sigma_{T_\alpha}\big(1_{\mathcal{B}_\alpha}\big)$  and $\sigma_{T^*_\alpha}\big(1_{\mathcal{B}_\alpha}\big)$ are contained in the circle of radius $e^{-1}$, that is,
	\begin{equation*}
		\sigma_{T_\alpha}\big(1_{\mathcal{B}_\alpha}\big) \subseteq \partial D(0,e^{-1}) \;\;\text{and}\;\;\sigma_{T^*_\alpha}\big(1_{\mathcal{B}_\alpha}\big) \subseteq \partial D(0,e^{-1}).
	\end{equation*}
\end{theorem}

\begin{proof}

	We will just prove the theorem for $T_\alpha$, an analogous argument works for $T_\alpha^*$. Let us denote the convergents of $\alpha$ by
$(a_j/q_j)_{j=0}^\infty$. Then, by Corollary \ref{co:norm_bound}, we know that for every $q_m \leq n^{2/3} \leq q_{m+1}$, we have
	\begin{equation*}
		\log||\widetilde{T}_\alpha^{-n}\, 1_{\mathcal{B}_\alpha}|| \leq C\cdot\left( q_m + \frac{n}{q_m}\log(q_m+1) + \frac{n+q_{m+1}}{q_{m+1}}\log(n+2)\right),
	\end{equation*}
	where $C> 0$ is an absolute constant independent of $m$. Taking into account the range of $n$, this implies
	\begin{equation*}
		||\widetilde{T}_\alpha^{-n}\, 1_{\mathcal{B}_\alpha}|| \leq \exp\left( \, C \cdot\left(n^{-1/3} + \frac{1}{q_m}\log(q_m+1) + n^{-2/3}\log(n+2)\right)\right)^n.
	\end{equation*}
	Nevertheless, for every $\varepsilon > 0$, there exists $m$ such that
	\begin{equation*}
		 C\cdot\left(n^{-1/3} + \frac{1}{q_m}\log(q_m+1) + n^{-2/3}\log(n+2)\right) \leq \varepsilon
	\end{equation*}
	for every $q_m \leq n^{2/3} \leq q_{m+1}$. In particular, as a consequence of this bound, we have that the function
	\begin{equation*}
		f_{1_{\mathcal{B}_\alpha}}(z)=\sum_{n=1}^\infty \big(\widetilde{T}_{\alpha}^{-n} \, 1_{\mathcal{B}_\alpha}\big)\cdot z^{n-1}
	\end{equation*}
	is analytic for $z \in D(0,e^{-\varepsilon})$. Since $f_{1_{\mathcal{B}_\alpha}}$ fulfills the equation $(\widetilde{T}_\alpha - zI)\,f_{1_{\mathcal{B}_\alpha}}(z) = 1_{\mathcal{B}_\alpha}$, this implies $D(0,e^{-\varepsilon})\subseteq \rho_{\widetilde{T}_\alpha} (1_{\mathcal{B}_\alpha})$. Finally, making $\varepsilon$ arbitrarily small, the theorem follows.
\end{proof}

As it is pointed out in \cite{Atzmon-1980}, since $T_\alpha$ has the SVEP, it may be seen that $\sigma_{T_\alpha}\big(1_{\mathcal{B}_\alpha}\big)$ coincides with the singular points within $\partial D(0,e^{-1})$ of its local resolvent function. This allows us to identify easily some of the basic properties which satisfy $\sigma_{T_\alpha}\big(1_{\mathcal{B}_\alpha}\big)$ (and therefore, $\sigma_{T_\alpha^*}\big(1_{\mathcal{B}_\alpha}\big)$ as well).
\begin{corollary}\label{COR-shape_local_spectrum}
	Let $\alpha\in (0,1)$ be any irrational number. Then, $\sigma_{T_\alpha} \big(1_{\mathcal{B}_\alpha}\big)$ (resp. $\sigma_{T_\alpha^*}\big(1_{\mathcal{B}_\alpha}\big)$) is symmetric with respect to the real axis and contains the point $\lambda = e^{-1}$.
\end{corollary}

\begin{proof}
We will just prove the result for $T_\alpha$. The first claim is a consequence of $\overline{f_{1_{\mathcal{B}_\alpha}(z)}} = f_{1_{\mathcal{B}_\alpha}}(\overline{z})$, where $f_{1_{\mathcal{B}_\alpha}}$ is as above. Note that it may be deduced from the fact that $\widetilde{T}_\alpha^{-n} \, 1_{\mathcal{B}_\alpha}$ are all real-valued for every $n\in\mathbb{Z}^+$.
	
	For the second claim, given any $z_0 \in \mathbb{D}$, the Taylor series of $f_{1_{\mathcal{B}_\alpha}}$ about $z_0$ is
	\begin{equation*}
		f_{1_{\mathcal{B}_\alpha}}(z) = \sum_{m=0}^\infty \frac{\partial^m f_{1_{\mathcal{B}_\alpha}}(z_0)}{m!} \cdot (z-z_0)^m,
	\end{equation*}
	where
	\begin{equation*}
		\partial^m f_{1_{\mathcal{B}_\alpha}}(z_0) = \sum_{n=m}^\infty n\cdot (n-1)\cdots (n-m+1) \cdot \big(\widetilde{T}_\alpha^{-n-1}\, 1_{\mathcal{B}_\alpha}\big) \cdot z_0^{n-m}.
	\end{equation*}
	Let $e^{\imath\vartheta}$ be a singular point on $\partial\mathbb{D}$ (there must exist at least one) for $f_{1_{\mathcal{B}_\alpha}}$ and choose any $0 < r < 1$. By hypothesis, the series
	\begin{equation*}
		f_{1_{\mathcal{B}_\alpha}}(z) = \sum_{m=0}^\infty \frac{\partial^m f_{1_{\mathcal{B}_\alpha}}(re^{\imath\vartheta})}{m!} \cdot (z-re^{\imath\vartheta})^m
	\end{equation*}
	has radius of convergence $1 - r$. Nevertheless, by the positivity of $\widetilde{T}_\alpha^{-n}$ and recalling that $1_{\mathcal{B}_\alpha}(t) \geq 0$ a.e., it may be seen that
	\begin{equation*}
		\big\lvert\big\lvert \partial^m f_{1_{\mathcal{B}_\alpha}}(re^{\imath\vartheta})\big\rvert\big\rvert \leq \big\lvert\big\lvert \partial^m f_{1_{\mathcal{B}_\alpha}}(r)\big\rvert\big\rvert \; \text{ for every } m \geq 0,
	\end{equation*}
	what, in particular, implies that the radius of convergence of the Taylor series of $f_{1_{\mathcal{B}_\alpha}}$ about $r$ cannot be greater than $1-r$. This proves that $1$ is a singular point for $\widetilde{T}_\alpha$ and the result follows.
\end{proof}
\begin{remark}
	The second part of the proof given for Corollary \ref{COR-shape_local_spectrum} is the vector-valued analogue of the classical result in Complex Analysis, known as Pringsheim Theorem; see, for example, \cite[Sec. 7.21]{titchmarsh}.
\end{remark}

In general, given an arbitrary operator $T \in \mathcal{L}(\mathcal{X})$, determining  the local spectrum at a non-zero $x \in \mathcal{X}$ is known to be a difficult problem.
Actually, finding vectors $x \in \mathcal{X}$ with non-trivial local spectra may be a hopeful starting point in order to seek for (hyper-)invariant subspaces, since the subsets of $\mathcal{X}$ defined as
\begin{equation*}
	X_T(F) = \big\lbrace x \in \mathcal{X} \, : \, \sigma_T(x) \subseteq F \big\rbrace,
\end{equation*}
turn out to be $T$-hyperinvariant linear manifolds by means of the properties \normalfont{(a), (b)} and \normalfont{(c)}, and behave well via functional calculus tools. They are called  \emph{local spectral subspaces} though they are not closed \emph{a priori}. Indeed, those operators $T\in\mathcal{L}(\mathcal{X})$ for which $X_T(F)$ is closed for every closed subset $F\subseteq \mathbb{C}$ are said to satisfy the \emph{Dunford property $(C)$}. Our next result states that
Bishop operators do not satisfy the \emph{Dunford property $(C)$}.

\begin{theorem}\label{THM-property_C}
	Let $\alpha \in (0,1)$ be any irrational. Then, the local spectral subspace $X_{T_\alpha}\big(\partial D(0,e^{-1})\big)$ (resp. $X_{T_\alpha^*}\big(\partial D(0,e^{-1})\big)$) is dense in $L^p[0,1)$ for $1 < p < \infty$. In particular, neither $T_\alpha$ nor  $T_\alpha^*$ have property $(C)$ on $L^p[0,1)$.
\end{theorem}
\begin{proof}
	Along the proof, let $M_\phi$ be the operator on $L^p[0,1)$ consisting on multiplying by $\phi$. As a direct consequence of the following identity
	\begin{equation}\label{EQ-unitary_equivalence}
		M_{e^{2\pi\imath m t}} \, T_\alpha =e^{-2\pi\imath m\alpha}\, T_\alpha \,  M_{e^{2\pi\imath m t}},
	\end{equation}
	we deduce that $T_\alpha$ and $e^{2\pi\imath m \alpha}T_\alpha$ are similar for every $m \in \mathbb{Z}$. In particular, this implies that
	\begin{equation*}
		\mathrm{span}\big\lbrace e^{2\pi\imath m t} \cdot 1_{\mathcal{B}_\alpha}(t) \, : \, m \in \mathbb{Z} \big\rbrace \subseteq X_{T_\alpha}\big(\partial D(0,e^{-1})\big),
	\end{equation*}
	and so, reminding that the span of the set $\{e^{2\pi\imath m t}\}_{m \in\mathbb{Z}}$ is dense in $L^p[0,1)$ for $1 < p < \infty$, we infer
	\begin{equation*}
		\big\lbrace x \in L^p[0,1) \, : \, \mathrm{supp}\, x \subseteq \mathcal{B}_\alpha\big\rbrace \subseteq \overline{X_{T_\alpha}\big(\partial D(0,e^{-1})\big)}.
	\end{equation*}
	Moreover, since ${T_\alpha} \big(X_{T_\alpha}\big(\partial D(0,e^{-1})\big)\big) = X_{T_\alpha}\big(\partial D(0,e^{-1})\big)$, we can try to perform the same argument with the set
	\begin{equation*}
	\mathrm{span}\big\lbrace T_\alpha \, M_{e^{2\pi\imath m t}} \,  1_{\mathcal{B}_\alpha} (t) \, : \, m \in \mathbb{Z} \big\rbrace = \mathrm{span}\big\lbrace t \, e^{2\pi\imath m t} \cdot 1_{\mathcal{B}_\alpha} (\{t+\alpha\}) \, : \, m \in \mathbb{Z} \big\rbrace.
	\end{equation*}
	But, since $M_t $ is a dense range operator, we deduce that the set $\{t \, e^{2\pi\imath m t}\}_{m \in\mathbb{Z}}$ spans densely within $L^p[0,1)$ again; hence
	\begin{equation*}
		\big\lbrace x \in L^p[0,1) \, : \, \mathrm{supp}\, x \subseteq \tau^{-1}_\alpha\big(\mathcal{B}_\alpha\big)\big\rbrace \subseteq \overline{X_{T_\alpha}\big(\partial D(0,e^{-1})\big)},
	\end{equation*}
	where $\tau_\alpha (t) = \{t+\alpha\}$. Now, as any operator of the form $M_{\{t+j\alpha\}}$ is of dense range and the finite product of dense range operators is again of this kind, we are in position to mimic our previous argument: for any $N\in\mathbb N$ we have
	\begin{eqnarray*}
		\overline{X_{T_\alpha}\big(\partial D(0,e^{-1})\big)} &\supseteq & \underset{j=-N, \ldots, N}{\mathrm{span}} \, \big\lbrace x \in L^p[0,1) \, : \, \mathrm{supp}\, x \subseteq \tau^{-j}_\alpha\big(\mathcal{B}_\alpha\big)\big\rbrace \\
		&= &\Big\lbrace x \in L^p[0,1) \, : \, \mathrm{supp}\, x \subseteq \bigcup_{j=-N}^N\tau^{-j}_\alpha\big(\mathcal{B}_\alpha\big)\Big\rbrace,
	\end{eqnarray*}
	so it also contains the set $W$ of $x$ with $\mathrm{supp}\, x\subset E=\bigcup_{j=-\infty}^{\infty}\tau^{-j}_\alpha\big(\mathcal{B}_\alpha\big)$. Since $\mathcal{B}_\alpha$ has strictly positive measure and $\tau_\alpha^{-1}$ is ergodic, we have that $E$ has measure 1 and therefore $W=L^p[0,1)$.

	Finally, for $T_\alpha^*$ an analogous argument works.
\end{proof}

\begin{remark}Observe that Theorem \ref{THM-property_C} applies not only for $T_\alpha$ or $T_\alpha^*$, but also for every non-invertible Bishop-type operator $T_{\phi,\alpha} \in \mathcal{L}\big(L^p[0,1)\big)$ which satisfies \cite[Thm. 2.6]{MACDONALD_inv_bishop_type}. In addition, this result somewhat complements the work begun in \cite{ARTICLE-gallardo_monsalve} consisting on identifying the local spectral properties fulfilled by those non-invertible Bishop-type operators.\medskip
\end{remark}

Finally, as a consequence of Theorem \ref{THM-property_C}, we show that the invariant linear manifolds consisting of the hyperrange, the analytical core and the algebraic core of $T_\alpha$ are dense in $L^p[0,1)$ for $1 < p < \infty$. Recall, the hyperrange of an operator $T\in\mathcal{L}(\mathcal{X})$ is defined as
\begin{equation*}
T^\infty(\mathcal{X}) = \bigcap_{n=1}^\infty T^n(\mathcal{X}).
\end{equation*}
In particular, by the injectivity of $T_\alpha$, we have that its hyperrange matches with its algebraic core $C(T_\alpha)$, which is defined as the greatest submanifold $M \subseteq L^p[0,1)$ such that $T_\alpha (M) = M$. On the other hand, the analytical core $K(T)$ is defined as the set of $x \in \mathcal{X}$ for which there exists a sequence $(x_n)_{n\geq 0}$ and a constant $\delta > 0$ such that
\begin{itemize}
	\item $x=x_0$ and $T x_{n+1} = x_n$ for every $n \geq 0$.
	\item $||x_n|| \leq \delta^n ||x||$ for every $n \geq 0$.
\end{itemize}
In general, one has that $K(T) \subseteq C(T)$ and $K(T)= X_T\big(\mathbb{C}\setminus \{0\}\big)$, see \cite[Thms. 1.21 and 2.18]{aiena}.
\begin{corollary}
	Let $\alpha \in (0,1)$ be any irrational number and $T_\alpha\in\mathcal{L}\big(L^p[0,1)\big)$ for $1 < p < \infty$. Then, all $K(T_\alpha)$, $C(T_\alpha)$ and $T_\alpha^{\infty}\big(L^p[0,1)\big)$ are non-closed dense linear submanifolds of $L^p[0,1)$. The same holds for $T_\alpha^* \in \mathcal{L}\big(L^p[0,1)\big)$ for $1 < p < \infty$.
\end{corollary}

\begin{proof}
Firstly, observe that $C(T_\alpha)$ is a dense linear submanifold of $L^p[0,1)$ since $T_\alpha$ is an injective dense range operator (see \cite[Corollary 1.B.3]{Beauzamy}). In addition, since $K(T_\alpha)$ trivially contains $X_{T_\alpha}\big(\partial D(0,e^{-1})\big)$, the result follows from Theorem \ref{THM-property_C}. An analogous proof works for $T_\alpha^*$.
\end{proof}

\bibliographystyle{plain}
\bibliography{ChGaMoUb}

\end{document}